\documentclass[12pt]{article}

\usepackage[a4paper, margin=1in]{geometry}
\usepackage[utf8]{inputenc}
\usepackage{amsmath}
\usepackage{amsfonts} 
\usepackage{amssymb}
\usepackage{amsthm}
\usepackage{subcaption}
\usepackage{graphicx}
\usepackage{pst-node}
\usepackage{float}
\usepackage{thmtools, thm-restate}
\usepackage{hyperref}
\usepackage[capitalize]{cleveref}

\graphicspath{ {./figures/} }

\newcommand{\Nn}{\mathbb{N}}

\newcommand{\Zz}{\mathbb{Z}}

\newtheorem{thm}{Theorem}

\theoremstyle{definition}
\newtheorem{defn}[thm]{Definition}
\newtheorem{rem}[thm]{Remark}
\newtheorem{lem}[thm]{Lemma}
\newtheorem{prop}[thm]{Proposition}
\newtheorem{cor}[thm]{Corollary}

\numberwithin{thm}{section}
\numberwithin{defn}{section}
\numberwithin{rem}{section}
\numberwithin{conj}{section}
\numberwithin{lem}{section}
\numberwithin{cor}{section}
\numberwithin{example}{section}
\numberwithin{prop}{section}
\numberwithin{figure}{section}

\allowdisplaybreaks

\begin{document}

\title{Expected Number of Dice Rolls Until an Increasing Run of Three}

\author{Daniel Chen}
\date{}

\maketitle

\begin{abstract}
    A closed form is found for the expected number of rolls of a fair n-sided die until three consecutive increasing values are seen. The answer is rational, and the greatest common divisor of the numerator and denominator is given in terms of n.
    
    As n goes to infinity, the probability generating function is found for the limiting case, which is also the exponential generating function for permutations ending in a double rise and without other double rises. Thus exact values are found for the limiting expectation and variance, which are approximately 7.92437 and 27.98133 respectively.
\end{abstract}

\section{Introduction}

The following problem is studied: What is $E_k(n)$, the expected number of rolls of a fair $n$-sided die until $k$ consecutive increasing values are seen?

When $k=2$, the solution is quickly derivable as $E_2(n) = \left(\frac{n}{n-1}\right)^n$, which approaches $e$ as $n \to \infty$. The $k=3$ case is more complex, and is the focus of this paper.

In \cref{sec:discrete}, it will be proven that $E_3(n)$ can be computed as follows. Define the sequence of functions $(a_i(x))_{i \in \Nn}$:
\[a_1(x) = x-1, a_2(x) = x(x-2), a_{r+2}(x) = (2x-1)a_{r+1}(x) - (x^2-x+1)a_r(x)\]
Then the $a_i(x)$ are polynomials with integer coefficients,  and $E_3(n) = \frac{n^n}{a_n(n)}$ (\cref{thm:disc_soln}). This motivates looking at the gcd of the numerator and denominator, and it is shown (\cref{thm:gcd}) that:
\[\gcd(n^n, a_n(n)) = \begin{cases}
    2^{1+\nu_2(\lfloor \frac{n}{12}\rfloor)}n^2 & \text{if } n \equiv 2 \text{ (mod 12)}\\
    n^2 & \text{if } n \equiv 5,8,11 \text{ (mod 12)}\\
    1 & \text{otherwise}
\end{cases}\]

Furthermore, the closed form of $E_3(n)$ is given (\cref{prop:closedform}), which implies that $E_3(n)$ can be computed in logarithmic time using repeated squaring. It is also proven that $E_3(n) \in \Nn \iff n=3$ (\cref{thm:e3_not_int}).

In \cref{sec:cont}, the limiting case will be focused on (as $n \to \infty$), which is equivalent to sampling real numbers between 0 and 1 until a run of three increasing values is seen. The probability generating function is found (\cref{thm:pgf}) to be:
\[P(x) = 1 + \frac{\sqrt3}{2}e^{\frac{x}{2}} (x-1) \sec\left(\frac{\pi}{6} + x\frac{\sqrt3}{2}\right)\]
where the coefficient of $x^r$ in $P(x)$ is the probability that the sampling process described ends on the $r$-th turn. Note that (other than the constant term) $P(x)$ is equal to the exponential generating function for the number of permutations of a given length with no increasing run of length 3 or greater, except for having one of length 3 at the end. This is a similar definition to David and Barton's generating function \cite{david62}, which counts the number of permutations with no increasing run of length 3 or greater.

$P(x)$ is found using the same combinatorial method as \cite{gessel14}. In this, the authors derive the exponential generating for permutations with all valleys even and all peaks odd, using a system of recurrence relations; the same is done here.

The first few coefficients of $\frac{x^n}{n!}$ in $P(x)$ are:
\begin{center}
\begin{tabular}{ c|c|c|c|c|c|c|c|c|c|c|c|c } 
$n$ & 0 & 1 & 2 & 3 & 4 & 5 & 6 & 7 & 8 & 9 & 10\\ 
\hline
$n!p(n)$ & 0 & 0 & 0 & 1 & 3 & 15 & 71 & 426 & 2778 & 20845 & 171729\\
\end{tabular}
\end{center}

Having found the generating function, the limiting expectation and variance are immediate (\cref{cor:mu_and_var}):
\[\mu = \frac{\sqrt{3e}}{\sqrt3C - S}\]
\[\text{Var} = \frac{9eC + \sqrt{3e}S - 3e}{(\sqrt3C-S)^2}\]
where $S=\sin\left(\frac{\sqrt3}{2}\right)$ and $C=\cos\left(\frac{\sqrt3}{2}\right)$.

\section{Solving for an $n$-sided die}\label{sec:discrete}

Let $E_3(n)$ be the expected number of rolls of a fair $n$-sided die until a run of three increasing values is seen. In this section, we will solve for $E_3(n)$ in terms of $n$.

\begin{defn}[$\mu_i$]
For each $i \in \Nn, 1\leq i \leq n$, let $\mu_i$ be the expected number of further rolls in the process given that the previous roll was $i$ (and we ignore any runs containing rolls before the previous roll).
\end{defn}

Note that $\mu_n = E_3(n)$, because having just rolled $n$ is identical to the start.
We may now form a system of simultaneous equations for each $\mu_i$, in order to find $\mu_n$.

\begin{prop}\label{prop:ui}
For each $1 \leq i \leq n$,
\[\mu_i = \frac{2n-i}{n} + \frac{2n-i}{n^2}\left(\sum_{j=1}^i\mu_j\right) + \frac{1}{n^2} \sum_{j=i+1}^n (n-i+1)\mu_j\]
\end{prop}

\begin{proof}
To calculate $\mu_i$, consider the situation where we have just rolled $i$ and it was a decrease. Consider our next roll $k$. Note that each case of $k$ occurs with probability $\frac{1}{n}$.
\begin{itemize}
    \item  For each $1 \leq k \leq i$, the process will be $(1+\mu_k)$ rolls long, because our roll did not increase so we reset.
    \item For each $i < k \leq n$, let us find the number of rolls. Consider cases for the third roll $r$, each of which happens with probability $\frac{1}{n}$. Note that we will definitely have at least two rolls, so we can count extra rolls after this, and add 2 to the result.
        \begin{itemize}
            \item For each $k+1 \leq r \leq n$, the process finishes and so the extra number of rolls after the first two is 0.
            \item For each $1 \leq r \leq k$, the process resets and so the extra number of rolls is $\mu_r$.
        \end{itemize} 
    Hence when $i < k \leq n$, by summing the probabilities multiplied by the outcomes, we obtain that for each of these cases the expected number of rolls is $2 + \sum_{r=1}^k \frac{1}{n}\mu_r$.
\end{itemize}

Now, summing the probabilities multiplied by the outcomes, we have that

\[\mu_i = \left(\sum_{k=1}^i\frac{1}{n}(1+\mu_k)\right) + \left( \sum_{k=i+1}^n \frac{1}{n}\left(2 + \sum_{r=1}^k \frac{1}{n}\mu_r\right)\right)\]

which we may then manipulate to the desired result:
\[\mu_i = \frac{1}{n}\left(i + \sum_{k=1}^i\mu_k\right) + \frac{1}{n}\left( \left(\sum_{k=i+1}^n 2\right) + \frac{1}{n}\sum_{k=i+1}^n\sum_{r=1}^k \mu_r\right)\]
\[= \frac{i}{n} + \frac{\sum_{k=1}^i\mu_k}{n} + \frac{2(n-i)}{n} + \frac{1}{n^2}\left(\left(\sum_{k=1}^n\sum_{r=1}^k \mu_r\right) - \left(\sum_{k=1}^i\sum_{r=1}^k \mu_r\right) \right)\]
\[= \frac{2n-i}{n} + \frac{\sum_{k=1}^i\mu_k}{n} + \frac{1}{n^2}\left(\left(\sum_{r=1}^n(n-r+1)\mu_r\right) - \left(\sum_{r=1}^i(i-r+1)\mu_r\right) \right)\]
\[= \frac{2n-i}{n} + \frac{n}{n^2}\left(\sum_{k=1}^i\mu_k\right) + \frac{1}{n^2}\left( \left(\sum_{r=1}^i((n-r+1)-(i-r+1))\mu_r\right) + \left(\sum_{r=i+1}^n(n-r+1)\mu_r\right) \right)\]
\[= \frac{2n-i}{n} + \frac{2n-i}{n^2}\left(\sum_{j=1}^i\mu_j\right) + \frac{1}{n^2}\left(\sum_{j=i+1}^n(n-j+1)\mu_j\right) \]
\end{proof}

\begin{defn}
For $n \geq 3$, define the $n \times n$ matrix $M_n$, and the $n \times 1$ vectors $V_n$ and $U_n$:
\[M_n:=
\begin{bmatrix}
n^2-1 & n-n^2-1 & 0 & 0 & 0 & \cdots & 0 & 0\\
-1 & n^2-1 & n-n^2-1 & 0 & 0 & \cdots & 0 & 0\\
-1 & -1 & n^2-1 & n-n^2-1 & 0 & \cdots & 0 & 0\\
-1 & -1 & -1 & n^2-1 & n-n^2-1 & \cdots & 0 & 0\\
-1 & -1 & -1 & -1 & n^2-1 & \cdots & 0 & 0\\
\vdots & \vdots & \vdots & \vdots & \vdots & \ddots & \vdots & \vdots\\
-1 & -1 & -1 & -1 & -1 & \cdots & n^2-1 & n-n^2-1\\
-n & -n & -n & -n & -n & \cdots & -n & n^2-n\\
\end{bmatrix}\]

\[V_n:=
\begin{bmatrix} n\\ n\\ \cdots\\ n\\ n^2\\ \end{bmatrix},
\;\; U_n := \begin{bmatrix} \mu_1\\ \mu_2\\ \cdots\\ \mu_n \end{bmatrix}\]

In other words,
\begin{figure}[H]
    \centering
    \begin{subfigure}[b]{0.5\textwidth}
    \centering
    \[m_{ij} = \begin{cases} n^2-n & \text{if } i=j=n\\ n^2-1 & \text{if } i=j, i\neq n\\ n-n^2-1 & \text{if } j=i+1\\ 0 & \text{if } j>i+1\\ -n & \text{if } i=n, j \neq n\\ -1 & \text{else} \end{cases}\]
    \end{subfigure}%
    \begin{subfigure}[b]{0.3\textwidth}
    \centering
    \[v_{i1} = \begin{cases} n^2 & \text{if } i=n\\ n & \text{else} \end{cases}\]
    \end{subfigure}%
    \begin{subfigure}{0.2\textwidth}
    \centering
    \[u_{i1} = \mu_i\]
    \end{subfigure}
\end{figure}

\end{defn}

\begin{prop}\label{lem:mateqn}
$M_n U_n = V_n$
\end{prop}

\begin{proof}
Multipling both sides of \cref{prop:ui} by $n^2$, we obtain:
\[n^2\mu_i = n(2n-i) + (2n-i)\left(\sum_{j=1}^i\mu_r\right) + \sum_{j=i+1}^n (n-j+1)\mu_j\]
\[\implies \left(\sum_{j=1}^{i-1}(i-2n)\mu_j\right) + ((n-1)^2 + i-1)\mu_i + \sum_{j=i+1}^n (j-1-n)\mu_j = 2n^2-in\]
\begin{equation}\label{eqn:convol}\implies \sum_{j=1}^{n}c_{ij}\mu_j = n(2n-i)\end{equation}
Where:
\[c_{ij} = \begin{cases}i-2n & \text{if } 1 \leq j \leq i-1\\ (n-1)^2+i-1 & \text{if } j=i\\j-1-n & \text{if } i+1 \leq j \leq n\end{cases}\]

Now, suppose that $1 \leq i < n$. We replace $i$ with $i+1$ in \cref{eqn:convol}:
\[\sum_{r=1}^{n}c_{i+1,r}\mu_r = n(2n-i-1)\]
This may be subtracted from \cref{eqn:convol}:
\begin{equation}\label{eqn:delta} \sum_{j=1}^{n}\underbrace{(c_{ij} - c_{i+1,j})}_{:=d_{ij}}\mu_j = n \end{equation}
When we consider cases for $d_{ij}$, it can be verified that $d_{ij} = m_{ij}$, i.e. $d_{ij}$ is equal to the corresponding entry in $M_n$.

Hence, when $i\neq n$, the LHS of \cref{eqn:delta} is equal to the $i$-th row of $M_nU_n$ (for all $i\neq n$), because $\mu_j=u_{j1}$. But the RHS of \cref{eqn:delta} is the $i$-th row of $V_n$ by definition, thus for $1 \leq i < n$, the $i$-th row of $M_nU_n$ and $V_n$ are equal.

Now when $i=n$, by \cref{eqn:convol} we deduce that:
\[\implies \sum_{j=1}^{n}c_{nj}\mu_j = n^2\]
Where $c_{nj}$ is $n^2-n$ if $j=n$, and $-n$ otherwise. Hence by definition, $c_{nj} = m_{nj}$, and so similarly to before, we have that the bottom rows of $M_nU_n$ and $V_n$ are equal. But we already showed that the $i$-th rows of $M_nU_n$ and $V_n$ are equal for all $1 \leq i < n$, therefore all rows of $M_nU_n$ and $V_n$ are equal, which implies the result.
\end{proof}

\begin{lem}\label{prop:detH}
let $H$ be the $(n-1) \times (n-1)$ matrix obtained by deleting the top row and rightmost column of $M_n$. Then:
\[\det(H) = (-1)^{n+1}n^{2n-3}\]
\end{lem}

\begin{proof}
By definition of $H$,
\[\det(H) = \det\left(
\begin{bmatrix}
-1 & n^2-1 & n-n^2-1 & 0 & 0 & \cdots & 0\\
-1 & -1 & n^2-1 & n-n^2-1 & 0 & \cdots & 0\\
-1 & -1 & -1 & n^2-1 & n-n^2-1 & \cdots & 0\\
-1 & -1 & -1 & -1 & n^2-1 & \cdots & 0\\
\vdots & \vdots & \vdots & \vdots & \vdots & \ddots & \vdots \\
-1 & -1 & -1 & -1 & \cdots & -1 & n^2-1\\
-n & -n & -n & -n & \cdots & -n & -n\\
\end{bmatrix}\right)\]

Now, for each row other than the bottom two, we subtract off the row below it. We also subtract $n$ copies of the penulimate row from the buttom row, and then swap the bottom two rows around:

\[\implies -\det(H) = \det\left(
\begin{bmatrix}
0 & n^2 & n-2n^2 & n^2-n+1 & 0 & \cdots & 0\\
0 & 0 & n^2 & n-2n^2 & n^2-n+1 & \cdots & 0\\
0 & 0 & 0 & n^2 & n-2n^2 & \cdots & 0\\
0 & 0 & 0 & 0 & n^2 & \cdots & 0\\
\vdots & \vdots & \vdots & \vdots & \vdots & \ddots & \vdots \\
0 & 0 & 0 & 0 & \cdots & n^2 & n-2n^2\\
0 & 0 & 0 & 0 & \cdots & 0 & -n^3\\
-1 & -1 & -1 & -1 & \cdots & -1 & n^2-1\\
\end{bmatrix}\right)\]

The entire first column of this matrix is zero except for the bottom entry, which is -1. Thus, let $H'$ be the $(n-2) \times (n-2)$ matrix obtained by deleting the first column and bottom row of $H$, then we have that:

\begin{equation}\label{eqn:detH}
    -\det(H) = (-1)^{n-1+1} \cdot (-1) \cdot \det(H')
\end{equation}

Note that $H'$ is upper triangular and thus its determinant is the product of the elements on its diagonal:
\[det(H') = \underbrace{n^2 \times n^2 \times \cdots \times n^2}_{n-3 \text{ times}} \times (-n^3)\]
\[= -n^{2n-3}\]
Substituting this into \cref{eqn:detH} gives the desired result.

\end{proof}

\begin{thm}\label{thm:mu_n}
\[\mu_n = \frac{n^{2n}}{\det(M_n)}\]
\end{thm}

\begin{proof}
Recall \cref{lem:mateqn}, whence we obtain $U_n = M_n^{-1} V_n$.

Note that $V_n$ can be written as:
\[V_n
=\begin{bmatrix} n\\ n\\ \cdots\\ n\\ n^2 \end{bmatrix}
=-n\left(
\underbrace{\begin{bmatrix}
n^2-1\\-1\\ \cdots\\ -1\\ -n
\end{bmatrix}}_{\text{leftmost col of }M_n}
-\begin{bmatrix}
n^2\\ 0\\ 0\\ \cdots\\ 0\\
\end{bmatrix}
\right)
\]
\[\implies V_n = -n(M_nI_{n\times1} - n^2I_{n\times1})\]
\[= (n^3I_{n\times n} - nM_n)I_{n\times1}\]

Where $I_{n\times n}, I_{n \times 1}$ are the $(n \times n)$ and $(n \times 1)$ identity matrices respectively. We may substitute this into $U_n = M_n^{-1}V_n$:

\[U_n = M_n^{-1}(n^3I_{n\times n} - nM_n)I_{n \times 1}\]
\[ = (n^3M_n^{-1} - nI_{n\times n})I_{n \times 1}\]

Now, $\mu_n$ is the bottom entry of $U_n$, by definition. Thus, from the above equation, $\mu_n$ is the bottom-left entry of $(n^3M_n^{-1} - nI_{n\times n})$, because multiplication by $I_{n \times 1}$ extracts the leftmost column. For convenience, write $BL(X)$ to mean the bottom-left entry of $X$:
\[\mu_n = BL(n^3M_n^{-1} - nI_{n\times n})\]
\[= n^3\cdot BL(M_n^{-1}) - n\cdot BL(I_{n\times n})\]
\begin{equation}\label{eqn:BL(m_inv)}
\therefore \mu_n = n^3\cdot BL(M_n^{-1})
\end{equation}

The bottom-left entry of $M_n^{-1}$ is equal to $\frac{1}{\det(M_n)}(-1)^{n+1} H_{1n}$ where $H_{1n}$ is the $(1,n)$ minor of $M_n$. $H_{1n}$ is equal to the determinant of the matrix obtained by deleting the top row and rightmost column of $M_n$; thus by \cref{prop:detH} we know that $H_{1n} = (-1)^{n+1}n^{2n-3}$. Substituting, we obtain that
\[BL(M_n^{-1}) = \frac{1}{\det(M_n)} (-1)^{n+1} (-1)^{n+1} n^{2n-3}\]
\[\implies n^3BL(M_n^{-1}) = \frac{n^{2n}}{\det(M_n)}\]
Together with \cref{eqn:BL(m_inv)}, this proves the theorem.

\end{proof}

\begin{defn} Define two sequences of polynomials as $A_2(x)=1$, $B_2(x)=0$, and for $i \geq 2$:
    \[A_{i+1}(x) = (x^2-1)A_i(x) - B_i(x)\]
    \[B_{i+1}(x) = (x^2-x+1)(A_i(x)+B_i(x))\]
\end{defn}

\begin{lem}\label{lem:sep_ab} For $i \geq 3$:
    \[\frac{A_{i+1}(x)}{x^{i+1}} = (2x-1)\frac{A_i(x)}{x^i} - (x^2-x+1)\frac{A_{i-1}(x)}{x^{i-1}}\]
    \[\frac{B_{i+1}(x)}{^{i+1}} = (2x-1)\frac{B_i(x)}{x^i} - (x^2-x+1)\frac{B_{i-1}(x)}{x^{i-1}}\]
\end{lem}

\begin{proof}
    By definition:
    \begin{equation}\label{eq:defnA}
    A_{i+1}(x) = (x^2-1)A_i(x) - B_i(x)
    \end{equation}
    \begin{equation}\label{eq:defnB}
    B_{i+1}(x) = (x^2-x+1)(A_i(x)+B_i(x))
    \end{equation}
    Multiplying \cref{eq:defnA} by $(x^2-x+1)$ and adding \cref{eq:defnB}:
    \[(x^2-x+1)A_{i+1}(x) + B_{i+1}(x) = (x^2-x+1)x^2A_i(x)\]
    Thus, shifting the index:
    \[-B_i(x) = (x^2-x+1)A_i(x) + (x^2-x+1)x^2A_{i-1}(x)\]
    Substituting this into \cref{eq:defnA} and rearranging, we obtain the required result for $A$. The corresponding result for $B$ can be derived similarly.
\end{proof}

\begin{prop}\label{prop:detM} $\det(M_n) = n^3(n-2)A_n(n) - n^3B_n(n)$\end{prop}

\begin{proof}
Let $n \in \Nn, n \geq 3$ be fixed.

For $r,k \in \Nn, r \geq 2$, define $f(r,k)$ to be the determinant of the following $r \times r$ matrix:

\[F(r,k) := \begin{bmatrix}
k & n-n^2-1 & 0 & 0 & 0 & \cdots & 0 & 0\\
-1 & n^2-1 & n-n^2-1 & 0 & 0 & \cdots & 0 & 0\\
-1 & -1 & n^2-1 & n-n^2-1 & 0 & \cdots & 0 & 0\\
-1 & -1 & -1 & n^2-1 & n-n^2-1 & \cdots & 0 & 0\\
-1 & -1 & -1 & -1 & n^2-1 & \cdots & 0 & 0\\
\vdots & \vdots & \vdots & \vdots & \vdots & \ddots & \vdots & \vdots\\
-1 & -1 & -1 & -1 & -1 & \cdots & n^2-1 & n-n^2-1\\
-1 & -1 & -1 & -1 & -1 & \cdots & -1 & n-1\\
\end{bmatrix}\]

For this matrix $F(r,k)$, $f_{ij} = \begin{cases}
k & \text{if } i=j=1\\
n-1 & \text{if } i=j=r\\
n^2-1 & \text{if } i=j, 1<i<r\\
n-n^2-1 & \text{if } j=i+1\\
0 & \text{if } j>i+1\\
-1 & \text{else}
\end{cases}$

Now, if we take $F(n,n^2-1)$ and multiply the bottom row by $n$ (which multiplies the determinant by $n$), then the result is precisely $M_n$, by definition. Hence:
\begin{equation}\label{eqn:mf}
    \det(M_n) = n \cdot f(n,n^2-1)
\end{equation}
\emph{Claim.} For $r \geq 3, f(r,k) = k \cdot f(r-1,n^2-1) + (n^2-n+1)f(r-1,-1)$
\begin{proof}
    In the first row of $F(r,k)$, all entries after the first two are zero. If we remove the top row and first column of $F(r,k)$, the new matrix we obtain is equal to $F(r-1, n^2-1)$. If we instead remove the top row and second column of $F(r,k)$, then the new matrix equals $F(r-1,-1)$. Using this to compute $\det(F(r,k))$, we have that $\det(F(r,k)) = k \det(F(r-1,n^2-1)) - (n-n^2-1)\det(F(r-1,-1))$, which implies the claim, by definition of $f(r,k)$.
\end{proof}

Now, recall the definition of $A_i(x)$ and $B_i(x)$, evaluating each of these at $x=n$:
\begin{equation}\label{eqn:Ann}
    A_2(n)=1, B_2(n)=0
\end{equation}
\begin{equation}\label{eqn:AnnRec}
    A_{i+1}(n) = (n^2-1)A_i(n) - B_i(n)
\end{equation}
\begin{equation}\label{eqn:BnnRec}
    B_{i+1}(n) = (n^2-n+1)(A_i(n)+B_i(n))
\end{equation}
\emph{Claim.} For each $ 2 \leq i \leq n$, the following is true:
\[f(n,n^2-1) = A_{n-i+2}(n) f(r,n^2-1) + B_{n-i+2}(n) f(r,-1)\]
\begin{proof}
    Induction on $i$. When $i=n$, it is true by \cref{eqn:Ann}. Now suppose it is true for $3 \leq i\leq n$, we will prove it for $i-1$. For brevity, let $X = f(i-1,n^2-1)$, $Y = f(i-1,-1)$.
    \[f(n,n^2-1) = A_{n-i+2} f(i,n^2-1) + B_{n-i+2} f(i,-1) \text{ by inductive hypothesis}\]
    \[= A_{n-i+2} \left((n^2-1)X + (n^2-n+1)Y\right)
    + B_{n-i+2} \left(-X + (n^2-n+1)Y\right) \text{ by the above claim}\]
    \[= ((n^2-1)A_{n-i+2}-B_{n-i+2})X + ((n^2-n+1)A_{n-i+2}+(n^2-n+1)B_{n-i+2})Y\]
    \[= A_{n-(i-1)+1}X + B_{n-(i-1)+2}Y \text{ by \cref{eqn:AnnRec} and \cref{eqn:BnnRec}}\]
\end{proof}

We are now ready to prove the desired result. Note that by definition, $f(2,k) = \det\left(
\begin{bmatrix}k & n-n^2-1\\ -1 & n-1 \end{bmatrix}
\right)$, thus $f(2,k) = k(n-1) - (n^2-n-1)$. Therefore, letting $k=n^2-1$ and $k=-1$ respectively, after simplifying we deduce that:
\[f(2,n^2-1) = n^2(n-2)\]
\[f(2,-1) = -n^2\]

Hence, by letting $i=2$ in the claim immediately above, it follows that $f(n,n^2-1) = A_n(n) \cdot n^2(n-2) + B_n(n) \cdot(-n^2)$. But by \cref{eqn:mf}, we have that $\det(M_n) = n \cdot f(n,n^2-1)$, and therefore $\det(M_n) = n^3(n-2)A_n(n) - n^3B_n(n)$, as required.

\end{proof}

\begin{defn}\label{defn:a}
    Define the sequence of polynomials $(a_i(x))_{i \in \Nn}$ as:
    \[a_1(x) = x-1\]
    \[a_2(x) = x(x-2)\]
    \[a_{i+2}(x) = (2x-1)a_{i+1}(x) - (x^2-x+1)a_i(x)\]
\end{defn}

\begin{lem}\label{lem:singlerec}
If the sequence of of functions $(b_i(x))_{i \geq 2}$ satisfies:
\[b_i(x) = x^3(x-2)\frac{A_i(x)}{x^i} - x^3\frac{B_i(x)}{x^i}\]
Then for each $2 \leq i \leq n, a_i(x) = b_i(x)$.
\end{lem}

\begin{proof}
    Let $(b_i(x))$ be defined as stated.
    Note that by \cref{lem:sep_ab}, the three sequences $(a_i(x))$, $\left(\frac{A_i(x)}{x^i}\right)$, $\left(\frac{B_i(x)}{x^i}\right)$ all satisfy the same second order linear recurrence, hence so does any linear combination of the latter two. But $(b_i(x))$ is such a linear combination, thus we only need to show that $(a_i(x))$ and $(b_i(x))$ are equal for $i=2$ and $i=3$ (because then the following terms will be equal by induction).

    When $i=2$, $a_i(x) = a_2(x) = x(x-2)$ by definition, and similarly, $b_i(x) = x^3(x-2)\frac{A_2}{x^2} - x^3\frac{B_2}{x^2} = x(x-2)$. Hence indeed, $a_2(x) = b_2(x)$.

    Similarly, when $i=3$, it can be verified that $a_3(x) = x^3-3x^2+1 = b_3(x)$.
\end{proof}

\begin{thm}\label{thm:disc_soln}
$E_3(n) = \frac{n^n}{a_n(n)}$.
\end{thm}

\begin{proof}
    Consider the sequence of functions $b_i(x) = x^3(x-2)\frac{A_i(x)}{x^i} - x^3\frac{B_i(x)}{x^i}$; by \cref{lem:singlerec}, the sequences $(a_i(x))$ and $(b_i(x))$ are equal for $2 \leq i \leq n$. Now recall \cref{prop:detM} and divide both sides by $n^n$:
    \[\frac{\det(M_n)}{n^n} = n^3(n-2)\frac{A_n}{n^n} - n^3\frac{B_n}{n^n}\]
    \[ = b_n(n) \text{ by definition}\]
    \[ = a_n(n) \text{ by \cref{lem:singlerec}}\]
\begin{equation}\label{eqn:detM=a}
    \therefore \det(M_n) = n^n a_n(n)
\end{equation}

Recall \cref{thm:mu_n}, which states that $\mu_n = \frac{n^{2n}}{\det(M_n)}$. In this we may substitute \cref{eqn:detM=a}:
\[\mu_n = \frac{n^{2n}}{n^n a_n(n)} = \frac{n^n}{a_n(n)}\]

Finally, note that $E_3(n) = \mu_n$ because by definition, $\mu_n$ is the expected number of further rolls given that the previous roll was $n$. But since the next roll must decrease, this situation is identical to the start. Hence $E_3(n) = \mu_n = \frac{n^n}{a_n(n)}$, as required.
\end{proof}

\begin{rem}
    The first few values of $a_n(n)$ are:
    \begin{center}
    \begin{tabular}{ c|c|c|c|c|c|c|c|c } 
    $n$ & 3 & 4 & 5 & 6 & 7 & 8 & 9 & 10\\ 
    \hline
    $a_n(n)$ & 1 & 15 & 225 & 3781 & 72078 & 1550016 & 37259191 & 991980099
    \end{tabular}
    \end{center}
\end{rem}

\begin{lem}\label{lem:closed} For a fixed value of $n$, the closed form for $a_i(n)$ is:
\[a_i(n) = \left(\frac{n-1}{2} - \frac{n+1}{\sqrt{-3}}\right)\left(n - \frac{1}{2} + \frac{\sqrt{-3}}{2}\right)^{i-1}
+ \left(\frac{n-1}{2} + \frac{n+1}{\sqrt{-3}}\right)\left(n - \frac{1}{2} - \frac{\sqrt{-3}}{2}\right)^{i-1}\]
\end{lem}

\begin{proof}
    The characteristic equation of the $a_i(n)$ is:
    \[\lambda^2 - (2n-1)\lambda + (n^2-n+1) = 0\]
    This has roots:
    \[\lambda_1, \lambda_2 = n-\frac{1}{2} \pm \frac{\sqrt{-3}}{2}\]
    And thus we may write $a_i(n) = X\lambda_1^{i-1} + Y\lambda_1^{i-1}$ where $X$ and $Y$ are constants to be determined. It can be verified that the first two terms of the supposed closed form match with $a_1(n)$ and $a_2(n)$, thus this determines $X$ and $Y$ to be the values given in the closed form, and we are done (by induction).
\end{proof}

\begin{prop}\label{prop:closedform}
\[E_3(n) = \frac{n^n}{ \left(\frac{n-1}{2} - \frac{n+1}{\sqrt{-3}}\right)\left(n - \frac{1}{2} + \frac{\sqrt{-3}}{2}\right)^{n-1}
+ \left(\frac{n-1}{2} + \frac{n+1}{\sqrt{-3}}\right)\left(n - \frac{1}{2} - \frac{\sqrt{-3}}{2}\right)^{n-1} }\]
\end{prop}

\begin{proof}
    Immediate, due to \cref{thm:disc_soln} and \cref{lem:closed}.
\end{proof}

\begin{rem}
    The closed form shows that $E_3(n)$ can be computed in logarithmic time, using repeated squaring.
\end{rem}

\begin{cor}[Continuous form of $E_3(x)$]\label{cor:continuousform} 
   \[E_3(x) = \frac{x^x}{\left(x^{2}-x+1\right)^{\frac{x-1}{2}}\left(\left(x-1\right)\cos\alpha-\frac{1}{\sqrt{3}}\left(x+1\right)\sin\alpha\right)}\]
   Where $\alpha = (x-1)\arctan\left(\frac{\sqrt{3}}{2x-1}\right)$.
\end{cor}
\begin{proof}
    In \cref{prop:closedform}, the complex roots $x-\frac{1}{2} \pm \frac{\sqrt{-3}}{2}$ can be written in exponential form. Then the expression can be manipulated into the result, using the exponential definition of $\sin$ and $\cos$.
\end{proof}

\begin{rem}
    $E_3(x)$ appears to be strictly decreasing, which would make sense since when more sides of the die are available, it should be easier to roll 3 increasing values. This could probably be proven by splitting $\frac{1}{E_3(x)}$ into a product and/or sum of simpler constituent functions, and showing that those constituent functions are increasing by differentiation.
\end{rem}

\begin{lem}\label{lem:an_mod_nnn}
Taking $a_n(n)$ mod $n^3$, we have:
\[a_n(n) \equiv \begin{cases}
    -\frac{n(n-1)}{2}n^2 + 1 & \text{ if } n \equiv 0 \text{ (mod 3)}\\
    n^2-1 & \text{ if } n \equiv 1 \text{ (mod 3)}\\
    n^2\frac{(n+1)(n-2)}{2} & \text{ else}
\end{cases}\]
\end{lem}

\begin{proof}
    Write $a_i(n)$ mod $n^3$ as $X_i n^2 + Y_i n + Z_i$, a polynomial in $\Zz[n]_{n^3}$. Then, substituting this into the recurrence relation that defines the $a_i$, it can be verified that:
    \[X_{i+1} = 2Y_i - X_i - X_{i-1} - Z_{i-1} + Y_{i-1}\]
    \[Y_{i+1} = 2Z_i - Y_i - Y_{i-1} + Z_{i-1}\]
    \[Z_{i+1} = -Z_i - Z_{i-1}\]
    With the base cases being $(X_1, Y_1, Z_1) = (0,1,-1), (X_2, Y_2, Z_2) = (1,-2,0)$.
    The closed forms for $X_i,Y_i,Z_i$ can then be proven by case bashing and induction. Writing mod as a binary operator, they are:
    \[Z_i = ((i-1) \text{\;mod\;} 3) - 1\]
    \[Y_i = i(((i+1) \text{\;mod\;} 3)-1)\]
    \[X_i = \frac{i(i-1)}{2}((i \text{\;mod\;} 3)-1)\]
    Letting $i=n$ and considering each case of $n \mod 3$, the result can be verified.
\end{proof}

\begin{lem}\label{lem:nu2(an)}
     If $n \equiv 2$ (mod 12), then:
     \[\nu_2(a_n(n)) = 3 + \nu_2\left(\Bigl\lfloor\frac{n}{12}\Bigr\rfloor\right)\]
     Where $\nu_2(x)$ denotes the exponent of the highest power of 2 dividing $x$.
\end{lem}

\begin{proof}
    It is enough to show that $\nu_2(a_2(n)) = \nu_2(a_{14(n)}) = 3 + \nu_2(\lfloor\frac{n}{12}\rfloor)$ and that if
    $\nu_2(a_{12i+2}(n)) = \nu_2(a_{12(i+1)+2}(n))$, then $\nu_2(a_{12(i+1)+2}(n)) = \nu_2(a_{12(i+2)+2}(n))$.
    This is because then we would have $3 + \nu_2(\lfloor\frac{n}{12}\rfloor) = \nu_2(a_2(n)) = \nu_2(a_{14}(n)) = \nu_2(a_{26}(n)) = \cdots = \nu_2(a_n(n))$, and so we would be done by induction.

    For convenience, write $n = 12k + 2$, then $\lfloor \frac{n}{12}\rfloor = k$.
    
    We first prove the base case.
    
    $\nu_2(a_2) = \nu_2(n(n-2))$ by definition, which equals $\nu_2(12k(12k+2)) = \nu_2(12) + \nu_2(k) + \nu_2(12k+2) = 3 + \nu_2(k)$.
    
    Now, by \cref{lem:closed}, it can be (tediously) verified that:
    \[a_{14}(n) = -14n + 91n^2 - 1001n^4 + 2002n^5 - 3432n^7 + 3003n^8 - 1001n^{10} + 364n^{11} - 14n^{13} + n^{14}
    \]
    Interestingly, this can be factorized to:
    \[a_{14}(n) =  n(n-2)(7-21n+35n^3-21n^4+n^6)(1-3n-12n^2+29n^3-3n^4-12n^5+n^6) 
    \]
    Therefore, since $a_2(n) = n(n-2)$ by definition, and $n$ is even by assumption, it follows that $a_{14}(n) = a_2(n) \times \text{Odd} \times \text{Odd}$. And so, $\nu_2(a_{14}(n)) = \nu_2(a_2(n) \times \text{Odd} \times \text{Odd}) = \nu_2(a_2(n))$, as required.

    Now we show the inductive step.

    Note that $a_2(n), a_{14}(n), a_{26}(n), \cdots, a_n(n)$ are evenly spaced terms of a second order linear recurrence, and thus these terms satisfy their own second order linear recurrence. Recall that by \cref{lem:closed}, the roots of the characteristic equation of the $a_i(n)$ are:
    \[\lambda_1, \lambda_2 = n-\frac{1}{2} \pm \frac{\sqrt{-3}}{2}\]
    Thus, $a_2(n), a_{14}(n), a_{26}(n), \cdots, a_n(n)$ satisfy the recurrence $a_{12(i+2)+2}(n) = Sa_{12(i+1)+2}(n) - Pa_{12i+2}(n)$, where:
    \[S = \lambda_1^{12} + \lambda_2^{12}\]
    \[P = \lambda_1^{12}\lambda_2^{12}\]

    $S$ and $P$ are polynomials in $n$ with integer coefficients. The constant term of $P$ is equal to the constant term of $(\underbrace{n^2-n-1}_{\lambda_1\lambda_2})^{12}$, which is 1. The constant term of $S$ is equal to $\left(-\frac{1}{2}+\frac{\sqrt{-3}}{2}\right)^{12} + \left(-\frac{1}{2}-\frac{\sqrt{-3}}{2}\right)^{12}$, which is $(e^{\frac{2\pi}{3}\sqrt{-1}})^{12} + (e^{-\frac{2\pi}{3}\sqrt{-1}})^{12}$ i.e. 2.

    Since $n$ is even by assumption, the parities of the constant terms imply that $S$ is even and $P$ is odd.

    \[ \therefore a_{12(i+2)+2}(n) = \text{Even} \times a_{12(i+1)+2}(n) + \text{Odd} \times a_{12i+2}(n)\]

    Now, by inductive hypothesis, we have $\nu_2(a_{12(i+1)+2}(n)) = \nu_2(a_{12(i+1)+2}(n))$. Hence, together with the preceding equation, this implies:
    
    \[ \therefore \nu_2(a_{12(i+2)+2}(n)) = \nu_2(\text{Even} \times a_{12(i+1)+2}(n) + \text{Odd} \times a_{12i+2}(n))\]
    \[= \nu_2(a_{12(i+1)+2}(n))\nu_2(\text{Even} + \text{Odd})\]
    \[= \nu_2(a_{12(i+1)+2}(n)) \text{ as required.}\]
    
\end{proof}

\begin{thm}\label{thm:gcd}
    \[\gcd(n^n, a_n(n)) = \begin{cases}
    2^{1+\nu_2(\lfloor \frac{n}{12}\rfloor)}n^2 & \text{if } n \equiv 2 \text{ (mod 12)}\\
    n^2 & \text{if } n \equiv 5,8,11 \text{ (mod 12)}\\
    1 & \text{otherwise} \end{cases}\]
\end{thm}

\begin{proof}
    The last case is proven by \cref{lem:an_mod_nnn}, because if $n \equiv 0,1$ (mod 3) then $a_n(n) \equiv \pm 1$ (mod $n$). Now suppose $n \equiv 2$ (mod 3); we will prove the other two cases.

    By \cref{lem:an_mod_nnn}:
    \begin{equation}\label{eqn:gcdthm_an_mod_nnn}
        a_n(n) \equiv n^2\frac{(n+1)(n-2)}{2} \text{ (mod $n^3$)}
    \end{equation}
    
    For convenience, write $\frac{(n+1)(n-2)}{2} = F$. Let us consider $\gcd(F,n)$, utilizing the euclidean algorithm. By assumption, $n \equiv 2,5,8,11$ (mod 12).
    \begin{itemize}
        \item If $n \equiv 5,11$, then $n$ is odd, and so $\gcd(F,n) = \gcd((n+1)(n-2),n) = \gcd(2,n) = 1$.
        \item If $n \equiv 2,8$, then $n$ is even; write $n=2k$, then $\gcd(F,n) = \gcd((2k+1)(k-1),2k) = \gcd(-k-1,2k) = \gcd(k+1,2)$. Hence if $n \equiv 8$ then the gcd is 1 because $\frac{n}{2}$ is even, and if $n \equiv 2$ then the gcd is 2.
    \end{itemize}

    Now by \cref{eqn:gcdthm_an_mod_nnn}, we may write $a_n(n) = Qn^3 + n^2F$ for some integer $Q$. Thus $\frac{a_n(n)}{n^2} = Qn + F$, where by the above, $\gcd(F,n)$ is 1 if $n \equiv 5,8,11$ (mod 12) and 2 if $n \equiv 2$ (mod 12).
    
    But $\gcd\left(Qn+F,n^{n-2}\right)$ divides $\gcd\left((Qn+F)^{n-2},n^{n-2}\right)$, which equals $(\gcd(Qn+F,n))^{n-2} = (\gcd(F,n))^{n-2} = 1^{n-2}$ or $2^{n-2}$. Therefore $\gcd\left(\frac{a_n(n)}{n^2}, n^{n-2}\right)$ is a power of $1$ or $2$, i.e. it is $1$ if $n \equiv 5,8,11$ and a power of $2$ if $n \equiv 2$.
    
    Hence the second case of the theorem is proven, and for the remaining $n \equiv 2$ case we know that $\gcd\left(\frac{a_n(n)}{n^2}, n^{n-2}\right)$ is a power of 2. From now on, we shall assume $n \equiv 2$ (mod 12). Since the aforementioned gcd is a power of 2, we have that:
    \[\gcd\left(\frac{a_n(n)}{n^2}, n^{n-2}\right) = 2^{\min(\nu_2\left(\frac{a_n(n)}{n^2}\right), \nu_2(n^{n-2}))}\]
    Where $\nu_2(x)$ denotes the exponent of the highest power of 2 dividing $x$.
    \[\implies n^2\gcd\left(\frac{a_n(n)}{n^2}, n^{n-2}\right) = n^2 \cdot 2^{\min(\nu_2(a_n(n)) - 2\nu_2(n), (n-2)\nu_2(n))}\]
    Note that $\nu_2(n) = 1$ since $n \equiv 2$ (mod 12):
    \[\implies \gcd(a_n(n), n^n) = 2^{\min(\nu_2(a_n(n)) - 2, n-2)}n^2\]
    Furthermore, by \cref{lem:nu2(an)} we know that $\nu_2(a_n(n))$ is equal to $3 + \nu_2\left(\lfloor \frac{n}{12}\rfloor\right)$:
    \[\implies \gcd(n^n, a_n(n)) = 2^{\min(1 + \nu_2(\lfloor \frac{n}{12}\rfloor), n-2)}n^2\]
    \[= n^2 \min(2^{1 + \nu_2(\lfloor \frac{n}{12}\rfloor)}, 2^{n-2})\]
    But $2^{1 + \nu_2(\lfloor\frac{n}{12}\rfloor)}$ is certainly smaller than $2^{n-2}$, because the former is at most $2 \cdot \frac{n}{12}$.
    \[\therefore \gcd(n^n, a^n) = 2^{1 + \nu_2(\lfloor \frac{n}{12}\rfloor)}n^2\]
    And the theorem is proven.
    
\end{proof}

\begin{thm}\label{thm:e3_not_int}
$E_3(n) \in \Nn \iff n=3$.
\end{thm}

\begin{proof}
    By \cref{thm:disc_soln}, $E_3(n)$ is given by $n^n / a_n(n)$. Also recall that, by \cref{thm:gcd}:
    \[\gcd(n^n, a_n(n)) = \begin{cases}
        2^{1+\nu_2(\lfloor \frac{n}{12}\rfloor)}n^2 & \text{if } n \equiv 2 \text{ (mod 12)}\\
        n^2 & \text{if } n \equiv 5,8,11 \text{ (mod 12)}\\
        1 & \text{otherwise}
    \end{cases}\]
    Note that if $n=3$, then $E_3(n) = 27$ is indeed an integer (since by computation, $a_3 = 1$ when $n=3$). Thus it remains to show the other direction, i.e. that $E_3(n) \in \Nn \implies n=3$.

    Suppose that $E_3(n) \in \Nn$, we will show that $n=3$.
    \[E_3(n) \in \Nn \implies \frac{n^n}{a_n(n)} \in \Nn\]
    \begin{equation}\label{eqn:gcd_eq_a}
    \implies \gcd(n^n, a_n(n)) = a_n(n)
    \end{equation}

    Let us consider cases; we will show that only one of them is possible.
    \begin{itemize}
        \item If $n \equiv 2$ (mod 3), let us show a contradiction. By \cref{lem:an_mod_nnn} and \cref{eqn:gcd_eq_a} it holds that:
        \[\gcd(n^n,a_n(n)) \equiv n^2\frac{(n+1)(n-2)}{2} \text{ (mod $n^3$)}\]
        by \cref{thm:gcd}, $\gcd(n^n,a_n(n)) = Cn^2$ where $C$ is $2^{1+\nu_2(\lfloor \frac{n}{12}\rfloor)}$ if $n \equiv 2$ (mod 12), and $1$ otherwise.
        \[\therefore Cn^2 \equiv n^2\frac{(n+1)(n-2)}{2} \text{ (mod $n^3$)}\]
        Hence there exists an integer $Q$ such that $ Cn^2 = Qn^3 + n^2\frac{(n+1)(n-2)}{2}$. This simplifies to $n^2 - n + 2Qn = 2C+2$, and so $n$ must divide $2C+2$, therefore $n \leq 2C+2$. But this is impossible, because by definition, $C$ is at most $2^{1+\nu_2(\lfloor \frac{n}{12}\rfloor)}$, which is at most $2 \cdot \frac{n}{12}$, and so $2C+2$ is at most $\frac{n}{3} + 2$. Hence this case cannot occur, i.e. $n \not\equiv 2$ (mod 3).

        \item If $n \equiv 0,1$ (mod 3), then let us show that $n=3$. By \cref{thm:gcd} and \cref{eqn:gcd_eq_a}, $a_n(n)$ must equal 1 (which is the gcd).
        
        This implies that by \cref{lem:an_mod_nnn} we have that $n$ must be $0$ mod 3 (because the other cases do not yield 1 when considering mod $n$ instead of $n^3$), with $a_n(n) \equiv -\frac{n(n-1)}{2}n^2 + 1$ (mod $n^3$). Hence $\exists Q \in \Zz$ with $1 = Qn^3 - \frac{n^3(n-1)}{2} + 1$. This implies that $Q = \frac{n-1}{2}$ and so $n$ must be odd. Since we already showed $n \equiv 0$ (mod 3), it follows that $n \equiv 3$ (mod 6).
        
        Now consider $a_n(n)$ mod $n-1$. Recall the definition of the $a_i(x)$, letting $x=n$:
        \[a_1 = n-1, a_2 = n(n-2), a_{i+2} = (2n-1)a_{i+1} - (n^2-n+1)a_i\]
        Thus, modulo $n-1$ we have:
        \[a_1 \equiv 0, a_2 \equiv -1, a_{i+1} \equiv a_{i+1} - a_i\]
        Listing out the first few terms of this sequence, they are $0, -1, -1, 0, 1, 1, 0, -1, \dots$ and so the period is $6$. We showed that $n \equiv 3$ (mod 6), therefore $a_n(n) \equiv -1$ (mod $n-1$) (because the third term of this periodic sequence is -1). However we also showed that $a_n(n) = 1$, therefore $-1 \equiv 1$ (mod $n-1$). It follows that $(n-1) \mid 2$, and thus $n=3$ as required, because we are only considering $n \geq 3$.
        
    \end{itemize}
    
\end{proof}

\section{Solving the limiting case}\label{sec:cont}

Let $P(x)$ be the probability generating function for the number of samples of a real number between $0$ and $1$ until 3 increasing values are seen. It this section, it will shown that the generating function is:
\[P(x) = 1 + \frac{\sqrt3}{2}e^{\frac{x}{2}}(x-1)\sec\left(\frac{\pi}{6} + x\frac{\sqrt3}{2}\right)\]
Note that, other than the constant term, $P(x)$ is equal to $F(x)$, the exponential generating function for permutations with exactly one increasing runs of length 3 or greater, which exactly consists of the last 3 elements (because any sequence of samples of length $n$ in $[0,1]$ can be flattened into a permutation of $1..n$). The constant terms are different because by convention, $F(0)=1$, whereas the probability that the sampling process takes zero turns is 0 (because it always takes at least three turns).

Similarly to \cite{gessel14}, we shall derive recurrence relations for $f(x)$ and two other helper functions $g(x)$ and $h(x)$, then we will use these relations to find $F(x)$, and thus $P(x)$.

\begin{defn}
    Let $F(x)$ be the exponential generating function for permutations with no increasing runs of length 3 or greater except for the last three which form an increasing run. Similarly, let $G(x)$ (resp. $H(x)$) be the exponential generating function for permutations with no increasing runs of length 3 or greater, and the last two (resp. one) elements form an increasing run. In other words, $G(x)$ ends in one increase, and $H(x)$ ends in a decrease.
    
    Define $f(n), g(n), h(n)$ in the usual way, to be the coefficients of $\frac{x^n}{n!}$ in their respective generating functions.
\end{defn}

\begin{lem}\label{lem:rec_fgh}[Recurrence relations for $f(x), h(x), h(x)$]
    For $n \geq 1$,
    \[f(n+1) = f(n) + g(n) + \sum_{k=1}^{n-1} {n \choose k}h(k)f(n-k)\]
    \[g(n+1) = g(n) + h(n) + \sum_{k=1}^{n-1} {n \choose k}h(k)g(n-k)\]
    \[h(n+1) = h(n) + 0 + \sum_{k=1}^{n-1} {n \choose k}h(k)h(n-k)\]
    Where the sums are zero if $n=1$. Also,
    \[f(0) = 1, f(1) = f(2) = 0\]
    \[g(0) = 1, g(1) = 0\]
    \[h(0) = 1\]
\end{lem}

\begin{proof}
    The initial conditions are true by convention. To calculate $f(n+1)$, take a permutation $\pi \in S_{n+1}$ and remove $n+1$ (the largest element). Since $n \geq 1$ by assumption, the first and last elements are different.

    \begin{itemize}
    \item If $n+1$ is the first element, then we have $f(n)$ ways, because the element following $n+1$ will definitely be a decrease.
    \item If $n+1$ is the last element ,then we have $g(n)$ ways, because the first $n$ elements must end in only one increase, so that adding $n+1$ yields two increases i.e. a run of length 3.
    \item If $n+1$ is not the first or last element, then remove it. This splits $\pi$ into two; let the left section be of length $k$. There are $n \choose k$ choices for the values in the left section; flattening these in the usual way yields $h(k)$ ways (because the element immediately following the left section is $n+1$ which will definitely be an increase, so the left section cannot end in an increase). Flattening the right section (of length $n-k$) yields $f(n-k)$ ways because the element following $n+1$ will definitely be a decrease. Thus, summing over the possible values of $k$, this case yields $\sum_{k=1}^{n-1} {n \choose k}h(k)f(n-k)$ ways.
    \end{itemize}
    
    Summing each case, there are $f(n+1) = f(n) + g(n) + \sum_{k=1}^{n-1} {n \choose k}h(k)f(n-k)$ ways, and so the recurrence for $f$ is proven. The other two recurrences can be derived in a similar way.
\end{proof}

\begin{lem}\label{lem:DEs}[Differential equations for $F(x), G(x), H(x)$]
    \[1 + F'(x) = G(x) - H(x) + H(x)F(x)\]
    \[1 + G'(x) = H(x)G(x)\]
    \[H'(x) = H(x)^2 - H(x) + 1\]
    \[F(0) = G(0) = H(0) = 1\]
\end{lem}

\begin{proof}
    The initial conditions follow from those of \cref{lem:rec_fgh}, because $F(0) = f(0)$ etc.
    Let us focus on the first equation.
    By \cref{lem:rec_fgh}, for $n \geq 1$ we have:
    \[f(n+1) = f(n) + g(n) + \sum_{k=1}^{n-1} {n \choose k}h(k)f(n-k)\]
    Since the sum was defined to be zero when $n=1$, and $h(0) = 1$, the $f(n)$ term can be absorbed into the sum:
    \[\implies f(n+1) = g(n) - h(n) + \sum_{k=0}^{n} {n \choose k}h(k)f(n-k)\]
    When $n=0$, the LHS is 0 and the RHS is $0 + 0 + h(0)f(0) = 1$.
    After multiplying by $x^n$ and summing over $n \geq 1$, we may add 1 to both sides to include the $n=0$ term:
    \[1 + \sum_{n \geq 1}f(n+1)\frac{x^n}{n!} = 1 + \sum_{n\geq 1}g(n)\frac{x^n}{n!} - \sum_{n\geq 1}h(n)\frac{x^n}{n!} + \sum_{n \geq 1}\frac{x^n}{n!}\sum_{k=0}^{n-1} {n\choose k}h(k)f(n-k) \]
    \[1 + \sum_{n \geq 0}f(n+1)\frac{x^n}{n!} = \sum_{n\geq 0}g(n)\frac{x^n}{n!} - \sum_{n\geq 0}h(n)\frac{x^n}{n!} + \sum_{n \geq 0}\frac{x^n}{n!}\sum_{k=0}^{n-1} {n\choose k}h(k)f(n-k) \]
    \[\implies 1 + F'(x) = G(x) - H(x) + \sum_{n\geq 0} \sum_{k=0}^n h(k)\frac{x^k}{k!}f(n-k)\frac{x^{n-k}}{(n-k)!} \]
    \[= G(x) - H(x) + H(x)F(x) \text{ as required.}\]

    The second and third equations are similarly derived.
    
\end{proof}

\begin{prop}\label{prop:DEsolution}
     \[H(x) = \frac{1}{2} + \frac{\sqrt3}{2}\tan\left(\frac{\pi}{6} + x\frac{\sqrt3}{2}\right)\]
     \[G(x) = \sec\left(\frac{\pi}{6} + x\frac{\sqrt3}{2}\right)\left(\frac{\sqrt3}{2}e^{x/2} - \sin\left(x\frac{\sqrt3}{2}\right)\right)\]
     \[F(x) = 2 + \frac{\sqrt3}{2}e^{\frac{x}{2}}(x-1)\sec\left(\frac{\pi}{6} + x\frac{\sqrt3}{2}\right)\]
\end{prop}

\begin{proof}
    Recall \cref{lem:DEs}:
    \begin{equation}\label{eqn:F'(x)}
        1 + F'(x) = G(x) - H(x) + H(x)F(x)
    \end{equation}
    \begin{equation}\label{eqn:G'(x)}
        1 + G'(x) = H(x)G(x)
    \end{equation}
    \begin{equation}\label{eqn:H'(x)}
        H'(x) = H(x)^2 - H(x) + 1
    \end{equation}
    \[F(0) = G(0) = H(0) = 1\]
    
    We first find $H(x)$ by solving \cref{eqn:H'(x)}, which is separable into $\frac{H'}{H^2-H+1} = 1$. Integrating both sides with respect to $x$, we deduce that $H(x) = \frac{1}{2} + \frac{\sqrt3}{2}\tan(c + x\frac{\sqrt3}{2})$ where $c$ is a constant to be determined. Then the initial condition $H(0)=1$ yields $c = \frac{\pi}{6}$, as required.

    Next we find $G(x)$ by solving \cref{eqn:G'(x)} with an integrating factor. The integral of $H(x)$ is $\frac{x}{2} - \ln(\cos(\frac{\pi}{6} + x\frac{\sqrt3}{2}))$; the rest of the derivation of $G(x)$ is computation.

    Finally, we find $F(x)$ by using the same integrating factor as before, because \cref{eqn:F'(x)} rearranges to $F'(x) -H(x)F(x) = G(x)-H(x)-1$.

    The calculations can be verified, but they are omitted for brevity.
\end{proof}

\begin{rem}
    The first few values enumerated by these generating functions are:
    \begin{center}
    \begin{tabular}{ c|c|c|c|c|c|c|c|c|c|c|c } 
    $n$ & 0 & 1 & 2 & 3 & 4 & 5 & 6 & 7 & 8 & 9 & 10\\ 
    \hline
    $f(n)$ & 1 & 0 & 0 & 1 & 3 & 15 & 71 & 426 & 2778 & 20845 & 171729\\
    $g(n)$ & 1 & 0 & 1 & 2 & 8 & 31 & 160 & 910 & 6077 & 45026 & 373220\\
    $h(n)$ & 1 & 1 & 1 & 3 & 9 & 39 & 189 & 1107 & 7281 & 54351 & 448821\\
    \end{tabular}
    \end{center}
\end{rem}

\begin{rem}
    Using \cref{lem:rec_fgh} to compute $\frac{f(n)}{g(n)}$ for large values of $n$, it appears to approach a limiting value of 0.4610896095... which we may prove and find a closed form for by using Theorem IV.7 of \cite{flajolet09}. We know that:
    \[\lim_{n \to \infty}\frac{f(n)/n!}{g(n)/n!} = \lim_{x\to a}\frac{F(x)}{G(x)}\]
    where $a = \frac{2\pi\sqrt3}{9}$ is the pole of both functions closest to the origin. This can be computed using \cref{prop:DEsolution} and multiplying the top and bottom by $\cos(\frac{\pi}{6} + x\frac{\sqrt3}{2})$; it can be verified that the resulting closed form is:
    \[\lim_{n\to\infty}\frac{f(n)}{g(n)} = \frac{a-1}{1-e^{-a/2}} \;\;\;,\; a=\frac{2\pi\sqrt3}{9}\]
    which matches the numerical value above. Similarly, $\lim_{n\to\infty}\frac{g(n)}{h(n)} = e^{a/2}-1$.
\end{rem}

\begin{rem}
    The values of $h(n)$ follow the sequence [A080635] on \cite{oeis}, for the number of permutations without double falls and without initial falls. This is as expected, since a valid permutation defined by $g(n)$ can be listed backwards and with the labels flipped.
\end{rem}

\begin{thm}\label{thm:pgf}
    The probability generating function for the number of samples of a real number between 0 and 1 until obtaining three increasing values is:
     \[P(x) = 1 + \frac{\sqrt3}{2}e^{\frac{x}{2}}(x-1)\sec\left(\frac{\pi}{6} + x\frac{\sqrt3}{2}\right)\]
\end{thm}

\begin{proof}
    By \cref{prop:DEsolution}, we have that the exponential generating function for the number of permutations that have exactly one run of length 3 or greater, that consists of exactly the last 3 elements, is:
     \[F(x) = 2 + \frac{\sqrt3}{2}e^{\frac{x}{2}}(x-1)\sec\left(\frac{\pi}{6} + x\frac{\sqrt3}{2}\right)\]
     
    Note that for $n \geq 1$, the coefficients of $x^n$ in $P(x)$ and $F(x)$ are equal. This is because the sampling process ending in $n$ turns corresponds to a permutation of length $n$, since the samples can be labelled $1$ to $n$ in chronological order - this results in a valid permutation counted by $F(x)$; and the probability of this happening is the number of valid permutations divided by $n!$. Also, the sampling process \emph{not} ending in $n$ turns corresponds to an \emph{invalid} permutation of length $n$, because we can chop off the sampling process or extend it to make it length $n$.

    Thus it remains to consider the coefficients of $x^0$. The probability of the process ending in 0 turns is 0, because it must take at least 3 turns. Hence $p(0) = 0$. By \cref{lem:rec_fgh}, we have that $f(0) = 1$, therefore $F(x)$ and $P(x)$ differ by 1; in conclusion, $P(x) = F(x) - 1$ and this implies the result.
\end{proof}

\begin{cor}\label{cor:mu_and_var}
    The limiting expectation and variance are:
    \[\mu = \frac{\sqrt{3e}}{\sqrt3C - S}\]
    \[\text{Var} = \frac{9eC + \sqrt{3e}S - 3e}{(\sqrt3C-S)^2}\]
    where $S=\sin\left(\frac{\sqrt3}{2}\right)$ and $C=\cos\left(\frac{\sqrt3}{2}\right)$. $\mu$ and Var are 7.9243724345... and 27.9813314059... respectively.
\end{cor}

\begin{proof}
    The expectation and variance are $P'(1)$ and $P''(1) + P'(1) - P'(1)^2$ respectively, which can be verified to be the stated values.
\end{proof}

\begin{rem}
    The limiting expectation can also be derived by computing the limit of the expression given for $E_3(n)$ in \cref{cor:continuousform}.
\end{rem}

\bibliographystyle{alpha}
\bibliography{bibliography}
\addcontentsline{toc}{section}{References}
\nocite{*}

\end{document}